\documentclass[a4,12pt]{article}
\usepackage{amsmath, amssymb, amscd, amsxtra, amsthm}
\usepackage{comment}

\theoremstyle{definition}
  \newtheorem{theorem}{Theorem}
  \newtheorem{lemma}[theorem]{Lemma}

  \newtheorem{remark}[theorem]{Remark}

\makeatletter

\@addtoreset{equation}{section}
\makeatother

\newcommand{\Z}{{\bf Z}}

\newcommand{\R}{{\bf R}}

\newcommand{\C}{{\bf C}}

\newcommand{\HH}{{\bf{H}}}

\newcommand{\pt}{\mathrm{pt}}

\newcommand{\tauu}{{\beta}}

\newcommand{\Ker}{\mbox{\rm Ker}\,}
\newcommand{\Coker}{\mbox{\rm Coker}\,}
\newcommand{\Spin}{\mbox{\rm Spin}\,}
\newcommand{\spin}{{\rm spin}\,}
\newcommand{\spinc}{{\rm spin}^c\,}

\newcommand{\rank}{\text{rank}\,}

\newcommand{\Ind}{\mbox{\rm Ind}\,}
\newcommand{\id}{1}

\newcommand{\dpin}{\Gamma}
\newcommand{\pin}{\mathrm{Pin^-(2)}}
\newcommand{\pinn}{\mathrm{Pin^-(1)}}
\newcommand{\Deltaa}{B}
\newcommand{\diff}{\omega}

\newcommand{\bysame}%
 {\leavevmode\hbox to 3em{\hrulefill}\,}

\title{Spin structures and the divisibility of Euler  classes}
\author{Yukio Kametani}

\begin{document}
\maketitle
\abstract{
In this short article
we give a geometric meaning of
the divisibility of $KO$-theoretical Euler classes
for given two spin modules.
We are motivated by 
Furuta's 10/8-inequality for a closed spin $4$-manifold.
The role of the reducibles
is clarified in the monopole equations of Seiberg-Witten theory,
as done by Donaldson and Taubes in Yang-Mills theory.
}

\section{Introduction}
Since Donaldson's celebrated work \cite{Donaldson},
the intersection form of spin $4$-manifolds
has been one of interests in gauge theory.
In the renewal of $4$-dimensional topology by Seiberg-Witten theory,
P.~B.~Kronheimer \cite{Kr} gave a lecture
on this problem following the method of Yang-Mills theory.
Soon after, M.~Furuta \cite{Furuta}
extracted an equivariant map from the monopole equation
to get the 10/8-inequality
through the divisibility of Euler classes in $K$-theory \cite{AtiyahTall}.
Although 
Furuta's approach is more sophisticated than Kronheimer's one,
its geometric meaning seems to have become vague.

In this paper we will build a bridge between
the geometry of the moduli spaces 
and the divisibility of the Euler classes.
Our argument seems to be a straightforward extension of Kronheimer's
method.
We bypass Furuta's finite dimensional approximation of the 
equation,
but use the local Kuranishi model of the reducibles.
The divisibility of the Euler classes
is a consequence of the compactness of the moduli space,
together with an equivariant spin structure on it.

We also investigate the divisibility when the first Betti number is positive
\cite{FK2}.
In this case we also need to take account of a natural map
from the moduli space to the Jacobian torus.
In \cite{FKMM} we have already used it to get
an invariant from the moduli space.
Related topics have been also discussed by H.~Sasahira \cite{Sasahira}.

The main result of this paper is a by-product of
the author's joint work with M.~Furuta on improvements of the 10/8-inequality 
\cite{FK2}.
This type of observation
originated from the three proofs of a special case of the inequality
in \cite{FKMa}.

This paper consists of two parts.
In Section 2 we give the main result,
which interprets the divisibility
in a general setting of cobordism theory.
To clarify the argument we divide it into three cases.
In Section 3 we set up gauge theory respecting $\pin$-symmetry
\cite{Kr}, \cite{Witten}
to get the divisibility \cite{Furuta}, \cite{FK2}
from our point of view.

The main theorem was announced 
at University of Minnesota in the summer of 2006.
The author thanks to T. J. Li for organizing the seminar,
and M. Furuta for his valuable suggestions and announcement of this result.
This paper was yielded from a collaborated work with him \cite{FK2}.
Some ideas have been already introduced in it.

\section{Main results}

Our main result is expressed by
a relation between
equivariant cobordism theory and $K$-theory.
In many variants of both theory, 
a specific case is applied to Seiberg-Witten theory.
However to clarify our argument,
we first explain a simpler version.

\subsection{Spin${}^c$ case}

We first recall the definition of equivariant $K$-theory.
Let $G$ be a compact Lie group.
For a compact Hausdorff $G$-space $B$,
we write the equivariant $K$-group as $K_G(B)$.
It is the Grothendieck group
of isomorphism classes of complex $G$-vector bundles over $B$.
For a (possibly) locally compact Hausdorff $G$-space $B$,
let
$B^+=B \cup \{+\}$ be the one point compactification of $B$.
Then 
$K_G(B)$ is the kernel of the map
$i^* : K_G(B^+) \to K_G(+)=R(G)$
induced from the inclusion $i: \{+\}\to B$.
More generally, for a closed $G$-subspace $C$ of $B$,
we define $K_G(B,C)=K_G(B \setminus C)$.
If we put $K^{-n}_G(B)=K_G(B \times \R^n)$,
where $\R^n$ is the $n$-dimensional trivial $G$-module,
we have equivariant cohomology theory for
locally compact Hausdorff $G$-spaces
(c.f. \cite{Karoubi}).
In the following argument we do not use cohomology theory
explicitly, but it helps us to achieve our results.

Let $V$ be a real $\spinc$ $G$-module of dimension $n$.
It means that 
the action on $V$ factors through
a given homomorphism
$G \to \Spin^c(n)=(\Spin(n) \times U(1))/\{ \pm 1\}$.

When $\dim V$ is even,
Bott periodicity theorem \cite{Atiyah}
tells us that 
the Bott class $\tauu(V) \in K_G(V)$ is
formed from the irreducible complex Clifford module for $V$
and $K_G(V)$ is freely generated  $\tauu(V)$ as an $R(G)$-module.
Then the Euler class $e(V) \in R(G)$
is defined to be the restriction of $\beta(V)$ to the zero $\{0\}$.

We next consider cobordism classes obtained from complex $G$-modules.
We first prepare some notations.
We take an auxiliary $G$-invariant norm on $V$.
Put $D(V)=\{v \in V \mid ||v|| \le 1\}$
and write its boundary as $S(V)$.
Write $E(V)=V \setminus \text{(interior of $D(V)$)}$.

Let $V_0$, $V_1$ be two real $\spinc$ $G$-modules.
We suppose that
\begin{enumerate}
\item[(i)] $\dim V_0 \ge \dim V_1 > 0$.
\item[(ii)] $\dim V_0 \equiv \dim V_1 \equiv 0 \pmod 2$.
\item[(iii)] the $G$-action on $V_0$ is free except the zero.
\end{enumerate} 
Then we can take  a $G$-map $\varphi: S(V_0) \to V_1$
which is transversal to the zero,
so that $\varphi^{-1}(0)$ is a closed 
free $G$-submanifold with
a $G$-isomorphism
$$ T \varphi^{-1}(0) \oplus \underline \R \oplus  
\underline {V_1} 
\cong \underline {V_0}.$$
Thus $\varphi^{-1}(0)$ has a unique $\spinc$ $G$-structure
such that the above isomorphism is a spin $G$-isomorphism,
where $\R$ is considered to be the trivial $G$-module $\R$.
This follows from the non-equivariant case (c.f. \cite{Kirby})
by taking the quotient space.

We let $\Omega_{G, \rm free}^{\spinc}$ be the
cobordism group of closed $\spinc$ free $G$-manifolds.
The cobordism class
$[\varphi^{-1}(0)] \in \Omega_{G,\rm free}^{\spinc}$
is independent of th choice of $\varphi$,.
since $V_1$ is $G$-contractible.
So we may write it as $\diff(V_0,V_1)$.


\begin{theorem}\label{spinc-th}
Let $V_0, V_1$ be real $\spinc$ $G$-modules satisfying
the condition (i), (ii), (iii) in the above.
If the cobordism class $\diff(V_0,V_1)
\in \Omega_{G, \rm free}^{\spinc}$ is zero,
there exists an element $\alpha$ in
$R(G)$ 
such that $$e(V_1) = \alpha e(V_0).$$
\end{theorem}

\begin{lemma}\label{G-ex}
Let $G$ be a compact Lie group,
and $V_0$ a real $G$-module.
Let $M$ be a compact $G$-manifold with boundary $\partial M$
and $i:\partial M \to S(V_0)$ a $G$-embedding.
Then there exist a real $G$-module $V$
and an $G$-embedding $i': M \to V_0 \oplus V$
such that $i'|\partial M=i$.
Moreover if $G$ acts on $M$ freely and $\dim  M  <  \dim V_0$,
we can take $i'$ to be
$i'(M) \subset E(V_0 \oplus V)\setminus E(0 \oplus V)$.
\end{lemma}

\begin{proof}
Using a $G$-color $\partial M \times [0,1]$ of $\partial M$,
we obtain a $G$-embedding
$i': \partial M \times [0,1] \to V_0 \times \R$
whose restriction to $\partial M \times \{0\}$ is $i$.
We next use \cite[Chapter VI]{Bredon}
to get a real  $G$-module $V'$
and a $G$-embedding $j: M \to V'$.
(One can use the double of M to reduce the case of closed $G$-manifolds.)
Let $\rho_k: M \to \R$ ($k=1,2$)
be a smooth $G$-invariant cut-off function such that
\begin{align*}
\begin{cases}
\rho_1(x)=1 &  x \in \partial M \times [0, \frac 2 3],\\
0 < \rho_1(x) <1 & x \in \partial M \times (\frac 2 3, \frac 5 6), \\
\rho_1(x)=0 &  x \notin \partial M \times [0, \frac 5 6],
\end{cases}
\quad
\begin{cases}
\rho_2(x)=0 &  x \in \partial M \times [0, \frac 1 6],\\
0 < \rho_2(x) <1 & x \in \partial M \times (\frac 1 6, \frac 1 3),\\
\rho_2(x)=1 &  x \notin \partial M \times [0, \frac 1 3].
\end{cases}
\end{align*}
Then
$i'=(\rho_1 i,  \rho_2 j, \rho_1-1, \rho_2):
M \to V_0 \oplus V'  \oplus \R^2$
is a desired $G$-embedding.

We also assume the additional condition.
Then $j(M) \cap \{0\}=\emptyset$,
and we may assume $j(M) \subset E(V')$
and $i'(M) \subset E(V_0 \oplus V'\oplus \R^2)$.
Moreover we can take $i'(M)$ to be  transversal to $V' \oplus \R^2$ in 
$V_0 \oplus V'  \oplus \R^2$.
The dimension condition implies
that $i'(M) \cap (0 \oplus V' \oplus \R^2) =\emptyset$.
\end{proof}

\noindent(Proof of Theorem \ref{spinc-th})
Pick a $G$-map $\varphi: S(V_0) \to V_1$ as above,
so that $[\varphi^{-1}(0)]=\diff(V_0,V_1)$.
By assumption we have a compact $\spinc$ free $G$-manifold $M$
with $\partial M =\varphi^{-1}(0)$ as $\spinc$ $G$-manifolds.
Let $i: \partial M \to V_0$ be the inclusion map.
We use Lemma \ref{G-ex} to find a $G$-module $V$ 
and a $G$-embedding $i': M \to E(V_0 \oplus V)\setminus E(0 \oplus V)$
such that $i'|\partial M=i$.
By adding some $G$-module, we may suppose that
$V$ is $\spinc$ $G$-module
and $\dim V \equiv 0 \pmod 2$.
For instance, $V \oplus V \cong V \otimes \C$ will do,
if we use 
the canonical $\Spin^c$ structure on $V \otimes \C$
(c.f. \cite{Atiyah-B-S}).

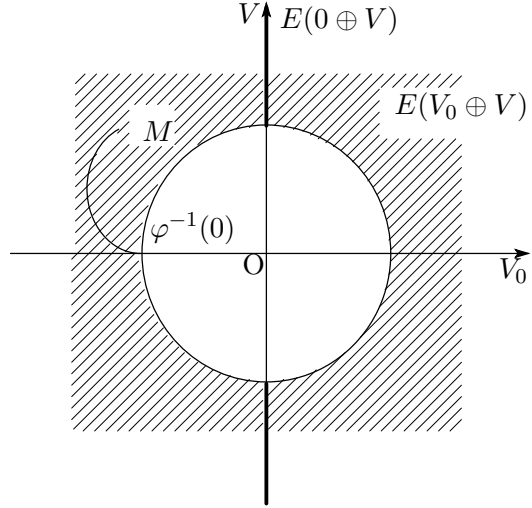
\begin{center}
{\unitlength 0.1in%
\begin{picture}(29.6000,26.2000)(7.6000,-30.8000)%
\put(20.9200,-17.7700){\makebox(0,0)[rt]{O}}%
\put(20.7300,-4.6000){\makebox(0,0)[rt]{$V$}}%
\put(34.7000,-17.9700){\makebox(0,0)[rt]{$V_0$}}%
%
\special{pn 8}%
\special{pa 2099 3080}%
\special{pa 2099 460}%
\special{fp}%
\special{sh 1}%
\special{pa 2099 460}%
\special{pa 2079 527}%
\special{pa 2099 513}%
\special{pa 2119 527}%
\special{pa 2099 460}%
\special{fp}%
%
\special{pn 8}%
\special{pa 760 1770}%
\special{pa 3470 1770}%
\special{fp}%
\special{sh 1}%
\special{pa 3470 1770}%
\special{pa 3403 1750}%
\special{pa 3417 1770}%
\special{pa 3403 1790}%
\special{pa 3470 1770}%
\special{fp}%
%
\special{pn 20}%
\special{pa 2099 1101}%
\special{pa 2099 480}%
\special{fp}%
%
\special{pn 20}%
\special{pa 2099 2452}%
\special{pa 2099 3080}%
\special{fp}%
\put(27.7100,-10.8100){\makebox(0,0)[lb]{$E(V_0\oplus V)$}}%
\put(21.7100,-6.2900){\makebox(0,0)[lb]{$E(0\oplus V)$}}%
%
\special{pn 8}%
\special{ar 2099 1770 651 673 0.4504814 0.4464672}%
%
\special{pn 4}%
\special{pa 1740 2340}%
\special{pa 1380 2700}%
\special{fp}%
\special{pa 1700 2320}%
\special{pa 1320 2700}%
\special{fp}%
\special{pa 1670 2290}%
\special{pa 1260 2700}%
\special{fp}%
\special{pa 1640 2260}%
\special{pa 1200 2700}%
\special{fp}%
\special{pa 1610 2230}%
\special{pa 1140 2700}%
\special{fp}%
\special{pa 1580 2200}%
\special{pa 1090 2690}%
\special{fp}%
\special{pa 1560 2160}%
\special{pa 1080 2640}%
\special{fp}%
\special{pa 1530 2130}%
\special{pa 1080 2580}%
\special{fp}%
\special{pa 1510 2090}%
\special{pa 1080 2520}%
\special{fp}%
\special{pa 1490 2050}%
\special{pa 1080 2460}%
\special{fp}%
\special{pa 1470 2010}%
\special{pa 1080 2400}%
\special{fp}%
\special{pa 1460 1960}%
\special{pa 1080 2340}%
\special{fp}%
\special{pa 1440 1920}%
\special{pa 1080 2280}%
\special{fp}%
\special{pa 1440 1860}%
\special{pa 1080 2220}%
\special{fp}%
\special{pa 1430 1810}%
\special{pa 1080 2160}%
\special{fp}%
\special{pa 1410 1770}%
\special{pa 1080 2100}%
\special{fp}%
\special{pa 1350 1770}%
\special{pa 1080 2040}%
\special{fp}%
\special{pa 1290 1770}%
\special{pa 1080 1980}%
\special{fp}%
\special{pa 1230 1770}%
\special{pa 1080 1920}%
\special{fp}%
\special{pa 1170 1770}%
\special{pa 1080 1860}%
\special{fp}%
\special{pa 1110 1770}%
\special{pa 1080 1800}%
\special{fp}%
\special{pa 1780 2360}%
\special{pa 1440 2700}%
\special{fp}%
\special{pa 1820 2380}%
\special{pa 1500 2700}%
\special{fp}%
\special{pa 1860 2400}%
\special{pa 1560 2700}%
\special{fp}%
\special{pa 1910 2410}%
\special{pa 1620 2700}%
\special{fp}%
\special{pa 1950 2430}%
\special{pa 1680 2700}%
\special{fp}%
\special{pa 2010 2430}%
\special{pa 1740 2700}%
\special{fp}%
\special{pa 2060 2440}%
\special{pa 1800 2700}%
\special{fp}%
\special{pa 2080 2480}%
\special{pa 1860 2700}%
\special{fp}%
\special{pa 2080 2540}%
\special{pa 1920 2700}%
\special{fp}%
\special{pa 2080 2600}%
\special{pa 1980 2700}%
\special{fp}%
\special{pa 2080 2660}%
\special{pa 2040 2700}%
\special{fp}%
%
\special{pn 4}%
\special{pa 3120 2100}%
\special{pa 2520 2700}%
\special{fp}%
\special{pa 3120 2040}%
\special{pa 2460 2700}%
\special{fp}%
\special{pa 3120 1980}%
\special{pa 2400 2700}%
\special{fp}%
\special{pa 3120 1920}%
\special{pa 2340 2700}%
\special{fp}%
\special{pa 3120 1860}%
\special{pa 2280 2700}%
\special{fp}%
\special{pa 3120 1800}%
\special{pa 2220 2700}%
\special{fp}%
\special{pa 3090 1770}%
\special{pa 2160 2700}%
\special{fp}%
\special{pa 2510 2290}%
\special{pa 2110 2690}%
\special{fp}%
\special{pa 2360 2380}%
\special{pa 2100 2640}%
\special{fp}%
\special{pa 2260 2420}%
\special{pa 2100 2580}%
\special{fp}%
\special{pa 2190 2430}%
\special{pa 2100 2520}%
\special{fp}%
\special{pa 3030 1770}%
\special{pa 2590 2210}%
\special{fp}%
\special{pa 2970 1770}%
\special{pa 2690 2050}%
\special{fp}%
\special{pa 2910 1770}%
\special{pa 2720 1960}%
\special{fp}%
\special{pa 2850 1770}%
\special{pa 2740 1880}%
\special{fp}%
\special{pa 2790 1770}%
\special{pa 2750 1810}%
\special{fp}%
\special{pa 3120 2160}%
\special{pa 2580 2700}%
\special{fp}%
\special{pa 3120 2220}%
\special{pa 2640 2700}%
\special{fp}%
\special{pa 3120 2280}%
\special{pa 2700 2700}%
\special{fp}%
\special{pa 3120 2340}%
\special{pa 2760 2700}%
\special{fp}%
\special{pa 3120 2400}%
\special{pa 2820 2700}%
\special{fp}%
\special{pa 3120 2460}%
\special{pa 2880 2700}%
\special{fp}%
\special{pa 3120 2520}%
\special{pa 2940 2700}%
\special{fp}%
\special{pa 3120 2580}%
\special{pa 3000 2700}%
\special{fp}%
\special{pa 3120 2640}%
\special{pa 3060 2700}%
\special{fp}%
%
\special{pn 4}%
\special{pa 2690 900}%
\special{pa 3720 900}%
\special{pa 3720 1170}%
\special{pa 2690 1170}%
\special{pa 2690 900}%
\special{ip}%
%
\special{pn 4}%
\special{pa 3120 1260}%
\special{pa 2740 1640}%
\special{fp}%
\special{pa 3120 1320}%
\special{pa 2740 1700}%
\special{fp}%
\special{pa 3120 1380}%
\special{pa 2750 1750}%
\special{fp}%
\special{pa 3120 1440}%
\special{pa 2790 1770}%
\special{fp}%
\special{pa 3120 1500}%
\special{pa 2850 1770}%
\special{fp}%
\special{pa 3120 1560}%
\special{pa 2910 1770}%
\special{fp}%
\special{pa 3120 1620}%
\special{pa 2970 1770}%
\special{fp}%
\special{pa 3120 1680}%
\special{pa 3030 1770}%
\special{fp}%
\special{pa 3120 1740}%
\special{pa 3090 1770}%
\special{fp}%
\special{pa 3120 1200}%
\special{pa 2730 1590}%
\special{fp}%
\special{pa 3090 1170}%
\special{pa 2710 1550}%
\special{fp}%
\special{pa 3030 1170}%
\special{pa 2700 1500}%
\special{fp}%
\special{pa 2970 1170}%
\special{pa 2680 1460}%
\special{fp}%
\special{pa 2910 1170}%
\special{pa 2660 1420}%
\special{fp}%
\special{pa 2850 1170}%
\special{pa 2630 1390}%
\special{fp}%
\special{pa 2790 1170}%
\special{pa 2610 1350}%
\special{fp}%
\special{pa 2730 1170}%
\special{pa 2580 1320}%
\special{fp}%
\special{pa 2690 1150}%
\special{pa 2550 1290}%
\special{fp}%
\special{pa 2690 1090}%
\special{pa 2520 1260}%
\special{fp}%
\special{pa 2690 1030}%
\special{pa 2490 1230}%
\special{fp}%
\special{pa 2690 970}%
\special{pa 2450 1210}%
\special{fp}%
\special{pa 2690 910}%
\special{pa 2410 1190}%
\special{fp}%
\special{pa 2710 830}%
\special{pa 2370 1170}%
\special{fp}%
\special{pa 2650 830}%
\special{pa 2330 1150}%
\special{fp}%
\special{pa 2590 830}%
\special{pa 2290 1130}%
\special{fp}%
\special{pa 2530 830}%
\special{pa 2240 1120}%
\special{fp}%
\special{pa 2470 830}%
\special{pa 2190 1110}%
\special{fp}%
\special{pa 2410 830}%
\special{pa 2140 1100}%
\special{fp}%
\special{pa 2350 830}%
\special{pa 2100 1080}%
\special{fp}%
\special{pa 2290 830}%
\special{pa 2100 1020}%
\special{fp}%
\special{pa 2230 830}%
\special{pa 2100 960}%
\special{fp}%
\special{pa 2170 830}%
\special{pa 2100 900}%
\special{fp}%
\special{pa 2770 830}%
\special{pa 2700 900}%
\special{fp}%
\special{pa 2830 830}%
\special{pa 2760 900}%
\special{fp}%
\special{pa 2890 830}%
\special{pa 2820 900}%
\special{fp}%
\special{pa 2950 830}%
\special{pa 2880 900}%
\special{fp}%
\special{pa 3010 830}%
\special{pa 2940 900}%
\special{fp}%
\special{pa 3070 830}%
\special{pa 3000 900}%
\special{fp}%
\special{pa 3120 840}%
\special{pa 3060 900}%
\special{fp}%
%
\special{pn 4}%
\special{ia 1440 2030 0 90 3.1415927 6.2831853}%
%
\special{pn 8}%
\special{ar 1440 1430 280 340 1.5707963 4.2934136}%
%
\special{pn 8}%
\special{pa 1380 990}%
\special{pa 1620 990}%
\special{pa 1620 1210}%
\special{pa 1380 1210}%
\special{pa 1380 990}%
\special{ip}%
\put(14.5000,-11.8000){\makebox(0,0)[lb]{$M$}}%
%
\special{pn 8}%
\special{pa 1100 830}%
\special{pa 3120 830}%
\special{pa 3120 2700}%
\special{pa 1100 2700}%
\special{pa 1100 830}%
\special{ip}%
%
\special{pn 4}%
\special{pa 1450 830}%
\special{pa 1100 1180}%
\special{fp}%
\special{pa 1510 830}%
\special{pa 1100 1240}%
\special{fp}%
\special{pa 1380 1020}%
\special{pa 1100 1300}%
\special{fp}%
\special{pa 1210 1250}%
\special{pa 1100 1360}%
\special{fp}%
\special{pa 1170 1350}%
\special{pa 1100 1420}%
\special{fp}%
\special{pa 1160 1420}%
\special{pa 1100 1480}%
\special{fp}%
\special{pa 1160 1480}%
\special{pa 1100 1540}%
\special{fp}%
\special{pa 1170 1530}%
\special{pa 1100 1600}%
\special{fp}%
\special{pa 1180 1580}%
\special{pa 1100 1660}%
\special{fp}%
\special{pa 1210 1610}%
\special{pa 1100 1720}%
\special{fp}%
\special{pa 1230 1650}%
\special{pa 1110 1770}%
\special{fp}%
\special{pa 1250 1690}%
\special{pa 1170 1770}%
\special{fp}%
\special{pa 1280 1720}%
\special{pa 1230 1770}%
\special{fp}%
\special{pa 1320 1740}%
\special{pa 1290 1770}%
\special{fp}%
\special{pa 1380 1080}%
\special{pa 1320 1140}%
\special{fp}%
\special{pa 1380 1140}%
\special{pa 1170 1350}%
\special{fp}%
\special{pa 1380 1200}%
\special{pa 1160 1420}%
\special{fp}%
\special{pa 1430 1210}%
\special{pa 1160 1480}%
\special{fp}%
\special{pa 1490 1210}%
\special{pa 1170 1530}%
\special{fp}%
\special{pa 1550 1210}%
\special{pa 1190 1570}%
\special{fp}%
\special{pa 1610 1210}%
\special{pa 1210 1610}%
\special{fp}%
\special{pa 2050 830}%
\special{pa 1230 1650}%
\special{fp}%
\special{pa 1530 1410}%
\special{pa 1260 1680}%
\special{fp}%
\special{pa 1480 1520}%
\special{pa 1290 1710}%
\special{fp}%
\special{pa 1450 1610}%
\special{pa 1320 1740}%
\special{fp}%
\special{pa 1440 1680}%
\special{pa 1360 1760}%
\special{fp}%
\special{pa 2080 860}%
\special{pa 1740 1200}%
\special{fp}%
\special{pa 2080 920}%
\special{pa 1860 1140}%
\special{fp}%
\special{pa 2080 980}%
\special{pa 1940 1120}%
\special{fp}%
\special{pa 2080 1040}%
\special{pa 2010 1110}%
\special{fp}%
\special{pa 1990 830}%
\special{pa 1620 1200}%
\special{fp}%
\special{pa 1930 830}%
\special{pa 1620 1140}%
\special{fp}%
\special{pa 1870 830}%
\special{pa 1620 1080}%
\special{fp}%
\special{pa 1810 830}%
\special{pa 1620 1020}%
\special{fp}%
\special{pa 1750 830}%
\special{pa 1590 990}%
\special{fp}%
\special{pa 1690 830}%
\special{pa 1530 990}%
\special{fp}%
\special{pa 1630 830}%
\special{pa 1470 990}%
\special{fp}%
\special{pa 1570 830}%
\special{pa 1410 990}%
\special{fp}%
\special{pa 1390 830}%
\special{pa 1100 1120}%
\special{fp}%
\special{pa 1330 830}%
\special{pa 1100 1060}%
\special{fp}%
\special{pa 1270 830}%
\special{pa 1100 1000}%
\special{fp}%
\special{pa 1210 830}%
\special{pa 1100 940}%
\special{fp}%
\special{pa 1150 830}%
\special{pa 1100 880}%
\special{fp}%
\put(14.9000,-17.2000){\makebox(0,0)[lb]{$\varphi^{-1}(0)$}}%
\end{picture}}%
\end{center}
Then the $G$-normal bundle $N$ to $M$ 
naturally inherits a $\spinc$ $G$-structure,
whose restriction to $\partial M$ is the one
induced by $\varphi$.
It implies that $\varphi_V^*\tauu(V_1 \oplus V)=j^*\tauu(N)$,
where $j:S(V_0 \oplus V) \rightarrow V_0 \oplus V$
is the inclusion
and $\varphi_V: S(V_0 \oplus V) \rightarrow V_1 \oplus V$
is the join of $\varphi$ and $\id_V$.
We thus get an element
\begin{align}\label{class}
\gamma \in 
K_{G}(E(V_0 \oplus V) \cup_{\varphi_V}(V_1 \oplus V),
E(0 \oplus V)\cup_{\id_{S(0 \oplus V)}} E(0 \oplus V)).
\end{align}
However, $\varphi_V$ is obviously $G$-homotopic 
to the map 
\begin{align*}
0_V : S(V_0 \oplus V) & \rightarrow D(V_1 \oplus V),\\
(u,v) & \mapsto (0,v).
\end{align*}
So we may replace
$\varphi_V$ in 
\eqref{class}
by  $0_V$.
The resulting space $E(V_0 \oplus V) \cup_{0_V}(V_1 \oplus V)$
is $G$-homeomorphic
to
$(V_0 \oplus V)  \cup_{\id_{D(0 \oplus V)}} (V_1 \oplus V)$,
since one can 
shrink the sphere $S(V_0)\times v$
into the point $(0,v)$ for each $v \in V$.
More explicitly
it is a $G$-map
from $E(V_0 \oplus V)$ to $V_0 \oplus V$
which takes the form of $(u,v) \mapsto (r(u,v)u,v)$, where
$$
r(u,v)
=\begin{cases}
\dfrac{||u||-\sqrt{1-||v||^2}}{||u||}, & ||v|| \le 1, \\
1, & ||v|| \ge 1.
\end{cases}
$$
So we may also suppose the class $\gamma$
is in
$$K_G((V_0 \oplus V)  \cup_{\id_{D(0 \oplus V)}} (V_1 \oplus V),
E(0 \oplus V)\cup_{\id_{S(0 \oplus V)}} E(0 \oplus V)).$$
Under the natural identification
$K_{G}(V_1 \oplus V, E(0 \oplus V))\cong K_{G}(V_1 \oplus V)$,
the restriction of $\gamma$ to $K_{G}(V_1 \oplus V)$
is $\beta(V_1 \oplus V)$,
while
its restriction to 
$K_{G}(V_0 \oplus V, E(0 \oplus V))\cong K_{G}(V_0 \oplus V)$
is $\alpha \beta(V_0 \oplus V)$
for some $\alpha \in K(\pt)$.
In $K_G(D(0 \oplus V), S(0 \oplus V)) \cong K_G(V)$
we have 
$ \alpha e(V_0)\tauu(V)= e(V_1)\tauu(V)$,
since the both sides are the same restriction of $\gamma$.



\subsection{Spin case}

The setting is similar to the previous one.
We denote by $KO_G(B)$ the equivariant $KO$-group,
that is, the Grothendieck group
of isomorphism classes of real $G$-vector bundles over $B$.
By putting $KO^{-m}_G(B)=KO_G(B \times \R^m)$
for the trivial $G$-module $\R^m$,
we have cohomology theory $KO^*_G(B)$.

Let $V$ be a real spin $G$-module of dimension $n$.
It means that 
the action on $V$ factors through
a given homomorphism
$G \to \Spin(n)$.
When $\dim V \equiv 0 \pmod 8$, 
$KO_G(V)$ is freely generated by the Bott class $\beta(V)$.
So we may write $KO^{n}_G(B)=KO_G(B \times \R^m)$,
if $n + m \equiv  0 \pmod 8$.
 
In \cite{FK2}
we extend Bott periodicity as follows:
Put $\tauu(V) =\tauu(V \oplus \R^m)\in KO_G^{n}(V)$
for $n + m \equiv  0 \pmod 8$.
Bott periodicity theorem indicates
that the total cohomology ring $KO^*(V)$
is freely generated by $\beta(V)$
as a $KO_G^*(\pt)$-module.
The Euler class $e(V)\in KO_G^{n}(\pt)$ 
is defined to be its restriction to $\{ 0 \} \oplus \R^m$.

For two real $\spin$ $G$-modules $V_0$, $V_1$
satisfying the condition (i), (iii) in the above,
we have a cobordism class
$\diff(V_0,V_1)$ 
in the cobordism group 
$\Omega_{G, \rm free}^{\spin}$
of closed $\spin$ free $G$-manifolds.
Our proof of the following theorem
is nothing but reputation.

\begin{theorem}\label{spin-th}
Let $V_0, V_1$ be real $\spin$ $G$-modules 
satisfying the condition (i), (iii) in the above.
If the cobordism class $\diff(V_0,V_1)
\in \Omega_{G, \rm free}^{\spin}$ is zero,
there exists an element $\alpha$ in
$KO^{d}_G(\pt)$ ($d=\dim V_1 -\dim V_0$)
such that $$e(V_1) = \alpha e(V_0).$$
\end{theorem}

Note that the condition (ii) is not necessary
in our extension of Euler classes.


\medskip

\subsection{Bundle cases}
We extend Theorem \ref{spin-th} to pairs of spin $G$-vector bundles.
Let $B$ be a compact Hausdorff $G$-space.
By a spin $G$-vector bundle $V$ over $B$,
we mean a spin structure on $V$,
together with a lift
of the $G$-action on $V$ to it.
When $\rank V \equiv 0 \pmod 8$,
$KO_G(V)$ is freely generated by the Bott class
$\beta(V)$ as a $KO_G(B)$-module.
For arbitrary $n=\rank V$,
put $\beta(V)=\beta(V \oplus \underline \R^m)
\in KO^n_G(V)$ for $n +m \equiv 0 \pmod 8$.
The Euler class $e(V) \in KO^n_G(B)$ is the restriction 
of $\beta(V)$
to $B \times \R^m$.
Here we identify the zero section with the base space $B$.


From now we assume that $B$ is a spin $G$-manifold.
Let $V_0$, $V_1$ be spin $G$-bundles over $\Deltaa$.
We suppose that the $G$ action is free on
$V_0 \setminus B$,
and so we can take a fiber-preserving $G$-map $\varphi: S(V_0) \to V_1$
which is transversal to the zero section $B$.
Then $\varphi^{-1}(B)$ 
is a closed free $G$-submanifold in $S(V_0)$.

Let $\pi: V_0 \to B$ be the projection.
We put a spin $G$-structure on $TV_0$
from an isomorphism $\pi^*(V_0 \oplus TB) \cong TV_0$
and on $\varphi^{-1}(0)$
from an isomorphism
\begin{align*}
T\varphi^{-1}(0) \oplus \underline \R 
\oplus  \pi^* V_1|\varphi^{-1}(0) \cong TV_0 |\varphi^{-1}(0).
\end{align*}
Let $\Omega_{G, \rm free}^{\spin}(B)$
be the cobordism group
of $G$-maps from a closed spin free $G$-manifold to $B$.
We may write the cobordism class of $\pi|\varphi^{-1}(0)$ 
as $\diff(V_0,V_1)\in \Omega_{G, \rm free}^{\spin}(B)$,
since it does not depend on the choice of $\varphi$.

\begin{theorem}\label{spin2}
Let $\Deltaa$ be a closed spin $G$-manifold.
Let $V_0, V_1$ be spin $G$-bundles over $B$.
satisfying the condition (i), (iii) 
on each fiber.
If $\diff(V_0,V_1)$ is zero in $\Omega_{G, \rm free}^{\spin}(\Deltaa)$,
there exists an element $\alpha \in KO^{d}_{G}(\Deltaa)$ 
($d=\rank V_1 -\rank V_0$)
such that $$e(V_1) = \alpha e(V_0).$$
\end{theorem}
\proof
We take a fiber-preserving $G$-map $\varphi: S(V_0) \to V_1$ as above,
so that $[\varphi^{-1}(0)]=\diff(V_0,V_1)$.
We can take a  $G$-bundle $V$ such
that $V_0 \oplus V \cong \underline V'$
for some real $G$-module $V'$ \cite{Segal}. 
Since the $G$-action on $\varphi^{-1}(0)$ is free,
we may use $G$-isotopy for the composite
$i :\varphi^{-1}(0) \to S(V_0) \to V_0 \oplus V  \to V'$
to be a $G$-embedding by adding some $G$-module to $V$.
By assumption we have a compact $\spinc$ free $G$-manifold $M$
and a $G$-map $\pi_M: M \to B$
with $\partial M =\varphi^{-1}(0)$ as $\spinc$ $G$-manifolds
and $\pi_M|\partial M =\pi|\varphi^{-1}(0)$.
It follows from Lemma \ref{G-ex}
that there exists a $G$-module $V''$
and a $G$-embedding  $i:  M \to V' \oplus V''$
such that $i'| \partial M =i$.
Then the product map $\pi_M \times i: M \to \Deltaa \times (V' \oplus V'')$
is a  $G$-embedding,
which commutes with projection to $\Deltaa$.
On the other hand, we may assume that
$V \oplus \underline V''$ is a spin $G$-bundle.
If one continues to take account of projection to $\Deltaa$,
the rest of the proof is very similar to Theorem \ref{spin-th}.
\qed

\section{Applications}
Let $X$ be a connected closed oriented spin $4$-manifold.
Take a Riemannian metric on $X$.
Then the spinor bundle for $X$ has a quaternionic structure.
We apply our result for the moduli space of monopoles,
which is known to be compact.
We denote by $W$ the framed moduli space,
on which $\pin=\langle U(1), j \rangle \subset \HH$
acts as the scalar multiplications on spinors of $X$,
and acts involution on forms of $X$
via the projection $\pin \to \{\pm 1\}$.
We may suppose that $W$ is smooth except the reducibles,
since the $\pin$-action is free  (c.f. \cite{Kr-M}).

If $b_1(X)=0$,
there is only one gauge equivalent classes of reducibles,
which is represented as the pair of 
the trivial connection and the zero section.
The Kuranishi model around it is
a $\pin$-equivariant map
\begin{align}\label{Kura-m}
\Phi: \HH^{k+p} \to \HH^{p} \oplus \tilde \R^{l}
\end{align}
where $k=-\sigma(X)/16$, $l=b^+_2(X)$
and $\tilde \R$ is the non-trivial one dimensional
real irreducible $\{\pm 1\}$-module.
We assume $l>0$.

When $l \equiv 0 \pmod 4$,
we can put a spin $\pin$-structure on $\tilde \R^l$.
Recall that the class $[T(W \setminus [\text{reducible}])]$
in $KO_{\pin}$-theory
is the index bundle of the linearization of the monopole equation,
which extends over any compact set of the ambient space (c.f. \cite{Mor}).
Since the ambient space is $\pin$-contractible 
to the reducible,
at which the index is $[\HH^k]-[\tilde \R^l] \in KO_{\pin}(\pt)$,
there is a $\pin$-isomorphism
$$T(W \setminus [\text{reducible}])
\oplus \underline{\tilde\R}^{l}\oplus \underline V
\cong \underline {\HH}^{k} \oplus  \underline V$$
for some $\pin$ -module $V$.
We may suppose that $V$ is a spin $G$-module.
Then it implies that $W \setminus [\text{reducible}]$
has a spin $\pin$-structure 
whose restriction to a neighborhood of the reducible
is that defined by \eqref{Kura-m}.
From Theorem \ref{spin-th},
we have the following divisibility of Euler classes
\begin{align}\label{div}
e(\HH^p \oplus \tilde \R^l)=\alpha e(\HH^{k+p})
\end{align}
for some $\alpha \in KO^{l-4k}_{\pin}(\pt)$.

When $l \equiv 0 \pmod 2$,
we can put a spin $\Gamma$-structure on $\tilde \R^l$,
and we have a divisibility \eqref{div}
in $KO^*_{\Gamma}(\pt)$,
where
$\dpin$ is the $2$-fold covering of $\pin$ 
obtained from the pull-back diagram:
\begin{equation*}
\begin{CD}
\dpin @>{}>> \pinn \\
   @VV{}V           @VV{}V    \\
\pin  @>{}>>             O(1),  \\
\end{CD}
\end{equation*}
and $\pinn$ is the pinor group $\mathrm{Pin(1)}$ in \cite{Atiyah-B-S}.
\begin{remark}
What we obtained here is a cobordism class of 
$G$-manifolds with 
$G$-equivariant stable parallelization,
which is stronger than spin structure.
Any corresponding cohomology theory is available to get divisibility,
if one knows Euler class.
\end{remark}

In general,
the reducibles consist
of pairs of flat $U(1)$-connection on the trivial bundle
and the zero section of the half spinor bundle.
We may identify it with the Jacobian torus
$$J_X=H^1(X; \R)/H^1(X;\Z).$$
We can associate each flat $U(1)$-connection $a$
to the Dirac operator $D_a$ on the spinor bundle.
We regard $J_X$ as a Real space. 
Let $Ksp(J_X)$ be the $Ksp$-group,
that is, the $K$-group of symplectic
bundles, equivalently,
complex vector bundles with
anti-complex linear action $j$ with $j^2=-1$.
Then the index bundle
$\Ind \Bbb D$ over $J_X$
constructed from the family $\Bbb D=\{D_a\}$
is an element in $Ksp(J_X)$,
We may write it as
$\Ind \Bbb D=[\Ker \Bbb D^+]-[\Ker \Bbb D^-]$.
Then the Kuranishi model around $J_X$ is
a fiber preserving $\pin$-equivariant map
$$\Phi: \Ker \Bbb D^+
\to \underline{H^+(X;\R)}\oplus \Coker \Bbb D^-.$$
We showed in \cite{FK2}
that one can put a spin $\pin$-structure 
on symplectic vector bundles over $J_X$.
Since the ambient space is $\pin$-contractible
to $J_X$.
Theorem \ref{spin-th} tells us
the following divisibility of Euler classes
\begin{align}\label{div-J}
e(\Ker \Bbb D^- \oplus \underline{H^+(X;\R)})
=\alpha e(\Ker \Bbb D^+)
\end{align}
for some $\alpha$ 
in $KO^{l-4k}_{\pin}(J_X)$,
or $KO^{l-4k}_{\dpin}(J_X)$.

\begin{remark}
The above divisibility may depend on the representation
of the index bundle 
as $\Ind \Bbb D=[\Ker \Bbb D^+]-[\Ker \Bbb D^-]\in Ksp(J_X)$.
Thus $\alpha$ is determined in the localization
in $KO^{l-4k}_{\pin}(J_X)_{\HH}$,
or $KO^{l-4k}_{\dpin}(J_X)_{\HH}$.
Our calculation in \cite{FK2} shows
that  the equation \eqref{div-J} determines $\alpha$
in the above localization.
\end{remark}

  {\scriptsize Department of Mathematics,
  Keio University, Yokohama 223-8522, Japan
  (kametani@math.keio.ac.jp)}

\end{document}